\documentclass[a4paper,UKenglish,cleveref, autoref, thm-restate]{lipics-v2021}

\usepackage{amsmath, amsfonts, amssymb, amsthm, bbm}
\usepackage{mathtools}
\usepackage{enumerate}
\usepackage{hyperref}
\usepackage{todonotes} 
\usepackage{tikz}
\usetikzlibrary{shapes.geometric, positioning, backgrounds, calc}
\tikzset{source/.style={circle,fill=gray!70,draw,minimum size=0.5cm,inner 
		sep=0pt}}
\tikzset{non-source/.style={circle,draw,minimum size=0.5cm,inner 
		sep=0pt}}

\DeclarePairedDelimiter\ceil{\lceil}{\rceil}
\DeclareMathOperator{\ecc}{ecc}
\DeclareMathOperator{\rad}{rad}
\DeclareMathOperator{\diam}{diam}
\newcommand{\Bound}{\ceil{\sqrt{n}~}}
\newcommand{\BoundC}{\ceil{\sqrt{cn}~}}
\newcommand{\SK}[1]{\textcolor{black}{#1}}
\newcommand{\YM}[1]{\textcolor{black}{#1}}
\newcommand{\NV}[1]{\textcolor{black}{#1}}
\newcommand{\rev}[1]{\textcolor{black}{#1}}

\title{Burning Graph Powers and Branching Trees} 

\author{Jesper {Jansson}}{Graduate School of Informatics, Kyoto University, Japan}{jj@i.kyoto-u.ac.jp}{https://orcid.org/0000-0001-6859-8932}{JST ASPIRE grant JPMJAP2302}

\author{Shashanka {Kulamarva}}{Graduate School of Informatics, Kyoto University, Japan}{kulamarva.shashanka.3k@kyoto-u.ac.jp}{https://orcid.org/0009-0002-2982-6044}{JST ASPIRE grant JPMJAP2302}

\author{Yukihiro {Murakami}}{Delft Institute of Applied Mathematics, TU Delft, the Netherlands}{Y.Murakami@tudelft.nl}{https://orcid.org/0000-0003-1355-5884}{}

\author{Nikolaas {Verhulst}}{Delft Institute of Applied Mathematics, TU Delft, the Netherlands}{N.D.Verhulst@tudelft.nl}{https://orcid.org/0000-0002-1631-0228}{}

\authorrunning{J. Jansson, S. Kulamarva, Y. Murakami, and N. Verhulst} 

\Copyright{Jesper Jansson, Shashanka Kulamarva, Yukihiro Murakami, and Nikolaas Verhulst} 



\category{Invited Talk} 

\relatedversion{} 

\nolinenumbers 

\EventEditors{Michal Kouck\'{y} and Daniela Petri\c{s}an}
\EventNoEds{2}
\EventLongTitle{51st International Symposium on Mathematical Foundations of Computer Science (MFCS 2026)}
\EventShortTitle{MFCS 2026}
\EventAcronym{MFCS}
\EventYear{2026}
\EventDate{August 24--28, 2026}
\EventLocation{Paris, France}
\EventLogo{}
\SeriesVolume{386}
\ArticleNo{12}

\begin{document}

\maketitle

\begin{abstract}
    Graph burning is a discrete-time process that models the spread of social contagion.
    Initially, all vertices are unburned.
    In each round, one unburned vertex is selected and burned, while any unburned vertex that has a burned neighbour from the previous round also becomes burned.
    The burning number of a graph is the minimum number of rounds needed to burn the entire graph.
    In this paper, we study the burning number of graph powers.
    First, we show that for a connected graph~$G$, its graph power~$G^k$ contains a~$(k+1)^+$-branching tree as a spanning tree.
    A~$(k+1)^+$-branching tree is one in which all internal vertices have degree at least~$k+1$.
    We then show that $(k+1)^+$-branching trees on~$n$ vertices have burning number at most $\left\lceil{\sqrt{\frac{4(k-1)n}{k^2}}}~\right\rceil$.
    As the burning number of a graph is at most the burning number of any of its spanning trees, this gives an upper bound on the burning number of graph powers.
    We also derive an alternative upper bound on the burning number of~$k^+$-branching trees using the strongest currently known general burning number bound [Bastide et al.]. 
    We then identify the ranges of~$k$ and~$n$ for which our bound outperforms or matches this alternative bound.
    Finally, we show that~$b(G^k) \le (1+o(1))\sqrt{n/k}$ based on the asymptotic burning number bound of Norin and Turcotte.
\end{abstract}

\section{Introduction}\label{sec:Intro}
\SK{\emph{Graph burning} is a discrete-time process introduced by Bonato et al.~\cite{bonato2016burn} to model the spread of contagion, influence, or information in a network.
Given a graph $G$, the process evolves in discrete rounds.
In each round, one chooses \rev{a} new vertex to be burned, while simultaneously the fire spreads from every vertex burned in earlier rounds to all of its neighbours.
Once a vertex is burned, it remains burned for the remainder of the process.
The burning number of $G$, denoted $b(G)$, is the minimum number of rounds required to burn all vertices of $G$.
The formal definition of graph burning is provided in \Cref{sec:Prelim}.}

\SK{Although graph burning was formalised as a graph parameter only recently, the same underlying problem was studied earlier in 1992 by Alon~\cite{alon1992cube} for hypercubes, in a communication-theoretic setting.
Graph burning is also closely related to several well-studied propagation processes on graphs,
such as the firefighter problem~\cite{finbow2009firefighter}}
\YM{and the~$k$-center problem~\cite{hakimi1964optimum}}.

\SK{
Graph burning has attracted considerable attention in recent years.
Bonato et al.~\cite{bonato2016burn} proved that every connected graph $G$ on $n$ vertices satisfies $b(G)\le 2 \Bound - 1$, and conjectured a tighter bound that is attained by paths.
\begin{conjecture}[\cite{bonato2016burn}]\label{conj:BurningNumber}
    If $G$ is a connected graph on $n$ vertices, then $b(G) \le \Bound$.
\end{conjecture}
This conjecture, commonly referred to as the \emph{Burning Number Conjecture}, remains one of the central open problems in the area.
A sequence of works has progressively improved the best-known general upper bound.
Following the initial bound of Bonato et al.~\cite{bonato2016burn}, Bessy et al.~\cite{bessy2018bounds} obtained a better upper bound}
\YM{($\sqrt{12n/3}+3$)} 
\SK{which was subsequently improved by Land and Lu}
\YM{($\lceil (\sqrt{24n+33}-3)/ 4\rceil$)}
\SK{~\cite{land2016burningbound} and later by Bastide et al.}
\YM{($\lceil \sqrt{4n/3}~\rceil + 1$)}
\SK{\cite{bastide2023improved}.
More recently, Norin and Turcotte~\cite{norin2024burning} proved that the conjecture holds asymptotically, i.e., $b(G) \le (1 + o(1)) \sqrt{n}$ for every connected graph $G$.}

\SK{To date, the conjecture has been settled in the affirmative for several tree classes, including paths~\cite{bonato2016burn}, spiders~\cite{bonato2019spiderpathforest}, $p$-caterpillars (for certain~$p$)~\cite{hiller2021caterpillar,liu2020caterpillar}},
\YM{and homeomorphically irreducible trees (HITs), which are trees whose internal vertices have degree at least~$3$~\cite{murakami2024burning}.}
\SK{This last result was recently extended to HITs on~$n$ vertices with at most $\lfloor \sqrt{n-1} \rfloor$ degree-$2$ vertices~\cite{ning2025burning}.}
\SK{For a survey of} \YM{other} \SK{recent results, see~\cite{bonato2021burningsurvey}.}

\YM{In this paper, we consider the burning of graph powers.}
\YM{To do so, we generalise the notion of HITs to~$k^+$-branching trees, which are trees whose internal vertices have degree at least~$k$.
We first show that for a tree~$T$, its graph power~$T^k$ contains a~$(k+1)^+$-branching \SK{spanning} tree (\Cref{lem:GraphPowerHask-angular}).
We then show that for~$k>1$, a~$k^+$-branching tree~$T$ on~$n$ vertices has burning number~$b(T)\le \left\lceil{\sqrt{\frac{4(k-2)n}{(k-1)^2}}}~\right\rceil$ (\Cref{thm:k-angBoundle}).
Since the burning number of a graph equals the minimum burning number of its spanning trees~\cite{bonato2016burn}, it follows that for a graph~$G$ on~$n$ vertices, 
\[b(G^k)\le \left\lceil\sqrt{\frac{4(k-1)}{k^2}n}~\right\rceil \quad\text{(\Cref{thm:GraphPowerBN})}.\]
In particular, this shows that any graph that is the square of some other graph satisfies~\Cref{conj:BurningNumber} (\Cref{cor:GraphSquareBN}).
In fact, since~$b(G^k)\le b(G^2)$, this shows that any non-trivial graph power satisfies \Cref{conj:BurningNumber}.
}

\YM{We also compare the burning number bound for~$k^+$-branching trees on~$n$ vertices of \Cref{thm:k-angBoundle} to the state-of-the-art bound of \cite{bastide2023improved}.
We first show that a~$k^+$-branching tree contains at least~$n(k-2)/(k-1)$ leaves (\Cref{lem:k-branchleaflowerbound}); by removing all leaves, adding an extra round, and using~\cite{bastide2023improved}, we show that~$k^+$-branching trees can be burned in at most~$\left\lceil \sqrt{\frac{4n}{3(k-1)}}~\right\rceil + 2$ rounds (\Cref{thm:Bastide4/3GP}).
We illustrate in \Cref{tab:UsvsBastide} the values of~$n$ and~$k$ where our bound from \Cref{thm:k-angBoundle} outperforms or matches that of~\Cref{thm:Bastide4/3GP}.
Finally, we show that the asymptotic bound on the burning number for~$k^+$-branching trees is $(1+o(1))\sqrt{n/(k-1)}$ (\Cref{thm:NorinGP})}.
\section{Preliminaries}\label{sec:Prelim}

For brevity, all graphs are assumed to be finite, simple, connected, and undirected, unless stated otherwise.
Given a graph~$G$, we write~$V(G)$ and~$E(G)$ to denote the vertex set and edge set of~$G$, respectively.
The \emph{degree} of a vertex~$v$, denoted~$\deg(v)$, is the number of edges incident to it.
By a \emph{tree}, we mean an acyclic graph. 
Within trees, vertices of degree $1$ \SK{are called leaves}; all other vertices are called \emph{internal vertices}.
Let~$k\in\mathbb{N}$.
A \emph{$k^+$-branching} tree is a tree in which all internal vertices are of degree at least~$k$.
By definition, all trees are $2^+$-branching.

Let~$G=(V,E)$ be a graph, and let~$u,v\in V$ be distinct vertices.
We write~$|G|=|V|$ to denote the number of vertices in~$G$.
The \emph{distance} between~$u$ and~$v$, denoted~$d_G(u,v)$, is the number of edges in a shortest path connecting them.
We remove the subscript and write~$d(u,v)$ when there is no ambiguity in the reference graph.
Note that in trees, such a shortest path is unique.
The \emph{eccentricity} of a vertex~$v$, denoted~$\ecc(v)$, is the greatest distance from~$v$ to any other vertex in~$G$, i.e., \[\ecc(v) = \max_{u\in V} d(u,v).\]
The \emph{radius} of~$G$, denoted~$\rad(G)$, is the minimum eccentricity over all vertices in~$G$, and the \emph{diameter} of~$G$, denoted~$\diam(G)$, is the maximum eccentricity over all vertices in~$G$, i.e.,
\begin{align*}
    & \rad(G) = \min_{v\in V} \ecc(v);\\
    & \diam(G) = \max_{v\in V} \ecc(v).
\end{align*} 
In other words, the diameter is the length of the longest shortest path between any two vertices.
For an integer~$k\ge0$, the \emph{$k$-neighbourhood} of a vertex~$v$, denoted $N^G_k[v]$, is the set of vertices that are at most distance $k$ away from~$v$, i.e., \[N^G_k[v] = \{u\in V: d(u,v) \le k\}.\]
In addition, the set $N^G_k[v] \setminus \{v\}$ is denoted by $N^G_k(v)$.
When there is no ambiguity in the reference graph, we shall simply write~$N_k[v]$ or $N_k(v)$.
For an integer~$k\ge1$, we write~$[k]:= \{1,2,\ldots, k\}$.

\SK{An edge in a graph is a \emph{bridge} (or \emph{cut-edge}) if its removal increases the number of connected components in the graph.}
Let~$xy$ be a bridge in~$G$.
We let~$G_{x}(xy)$ denote the component in which~$x$ resides, upon removing~$xy$ from~$G$.

Let~$k\in\mathbb{N}$. 
The \emph{$k$th power of~$G$} is the graph~$G^k=(V,E')$ where~\[E'= \{uv: u,v\in V, d_G(u,v) \le k\}.\]
In other words,~$G^k$ has the same vertex set as~$G$, with edges between vertices which are at distance at most~$k$ in~$G$.

\SK{We now define the process of graph burning.
It proceeds in discrete rounds.
\rev{Vertices can either be in a \emph{burned} or an \emph{unburned} state. Once a vertex becomes burned, it cannot be unburned. 
All vertices start unburned.}
In each round (time-step) \rev{$i\ge 1$}, a new vertex $b_i$, called a \emph{source}, is chosen to be burned, and simultaneously, all unburned neighbours of every vertex that was burned in step $i-1$ become burned.
Note that~$b_i$ could already be burned in round~$i$.
The process continues until all vertices of $G$ are burned.}
\SK{The sequence of selected sources $(b_1,b_2,\ldots,b_k)$ is called a \emph{burning sequence} for $G$.}
\SK{The \emph{burning number} of a graph $G$, denoted $b(G)$, is the minimum number of rounds required to burn all vertices of $G$.} 
See \Cref{fig:ExampleBurning} for an example of burning a graph.

\begin{figure}[ht!]
	\centering
	\resizebox{1\textwidth}{!}{%
	\begin{tikzpicture}
		\begin{scope}
			
			\node[non-source] (2) {};
            \node[draw = none, fill = none, above=1mm of 2]  (a) {\large $v_3$};
			\node[non-source] (1) [above left = 0.75cm of 2] {};
            \node[draw = none, fill = none, above=1mm of 1]  (a) {\large $v_1$};
			\node[non-source] (3) [below left = 0.75cm of 2] {};
            \node[draw = none, fill = none, below=1mm of 3]  (a) {\large $v_2$};
			\node[non-source] (4) [right = 0.5cm of 2] {};
            \node[draw = none, fill = none, above=1mm of 4]  (a) {\large $v_4$};
			\node[non-source] (5) [below = 0.4cm of 4] {};
            \node[draw = none, fill = none, below=1mm of 5]  (a) {\large $v_5$};
            \node[draw = none, fill = none, below=8mm of 5]  (a) {\large round $0$};
			\node[non-source] (6) [right = 0.5cm of 4] {}; 
            \node[draw = none, fill = none, above=1mm of 6]  (a) {\large $v_6$};
			\node[non-source] (7) [below = 0.4cm of 6] {};
            \node[draw = none, fill = none, below=1mm of 7]  (a) {\large $v_7$};
			\node[non-source] (8) [right = 0.5cm of 7] {};
            \node[draw = none, fill = none, below=1mm of 8]  (a) {\large $v_8$};
            \node[non-source] (9) [above = 0.4cm of 8] {};
            \node[draw = none, fill = none, above=1mm of 9]  (a) {\large $v_9$};
			
			\path[draw,thick]
			(1) edge node {} (2)
			(2) edge node {} (3)
			(2) edge node {} (4)
			(4) edge node {} (5)
			(4) edge node {} (6)
			(6) edge node {} (7)
			(6) edge node {} (8)
            (8) edge node {} (9);
		\end{scope}
		\begin{scope}[xshift=5cm]
			
			\node[source] (2) {$1$};
            \node[draw = none, fill = none, above=1mm of 2]  (a) {\large $v_3$};
			\node[non-source] (1) [above left = 0.75cm of 2] {};
            \node[draw = none, fill = none, above=1mm of 1]  (a) {\large $v_1$};
			\node[non-source] (3) [below left = 0.75cm of 2] {};
            \node[draw = none, fill = none, below=1mm of 3]  (a) {\large $v_2$};
			\node[non-source] (4) [right = 0.5cm of 2] {};
            \node[draw = none, fill = none, above=1mm of 4]  (a) {\large $v_4$};
			\node[non-source] (5) [below = 0.4cm of 4] {};
            \node[draw = none, fill = none, below=1mm of 5]  (a) {\large $v_5$};
            \node[draw = none, fill = none, below=8mm of 5]  (a) {\large round $1$};
			\node[non-source] (6) [right = 0.5cm of 4] {}; 
            \node[draw = none, fill = none, above=1mm of 6]  (a) {\large $v_6$};
			\node[non-source] (7) [below = 0.4cm of 6] {};
            \node[draw = none, fill = none, below=1mm of 7]  (a) {\large $v_7$};
			\node[non-source] (8) [right = 0.5cm of 7] {};
            \node[draw = none, fill = none, below=1mm of 8]  (a) {\large $v_8$};
            \node[non-source] (9) [above = 0.4cm of 8] {};
            \node[draw = none, fill = none, above=1mm of 9]  (a) {\large $v_9$};
			
			\path[draw,thick]
			(1) edge node {} (2)
			(2) edge node {} (3)
			(2) edge node {} (4)
			(4) edge node {} (5)
			(4) edge node {} (6)
			(6) edge node {} (7)
			(6) edge node {} (8)
            (8) edge node {} (9);
		\end{scope}
		\begin{scope}[xshift=10cm]
			
			\node[source] (2) {$1$};
            \node[draw = none, fill = none, above=1mm of 2]  (a) {\large $v_3$};
			\node[non-source] (1) [above left = 0.75cm of 2] {$2$};
            \node[draw = none, fill = none, above=1mm of 1]  (a) {\large $v_1$};
			\node[non-source] (3) [below left = 0.75cm of 2] {$2$};
            \node[draw = none, fill = none, below=1mm of 3]  (a) {\large $v_2$};
			\node[non-source] (4) [right = 0.5cm of 2] {$2$};
            \node[draw = none, fill = none, above=1mm of 4]  (a) {\large $v_4$};
			\node[non-source] (5) [below = 0.4cm of 4] {};
            \node[draw = none, fill = none, below=1mm of 5]  (a) {\large $v_5$};
            \node[draw = none, fill = none, below=8mm of 5]  (a) {\large round $2$};
			\node[source] (6) [right = 0.5cm of 4] {$2$}; 
            \node[draw = none, fill = none, above=1mm of 6]  (a) {\large $v_6$};
			\node[non-source] (7) [below = 0.4cm of 6] {};
            \node[draw = none, fill = none, below=1mm of 7]  (a) {\large $v_7$};
			\node[non-source] (8) [right = 0.5cm of 7] {};
            \node[draw = none, fill = none, below=1mm of 8]  (a) {\large $v_8$};
            \node[non-source] (9) [above = 0.4cm of 8] {};
            \node[draw = none, fill = none, above=1mm of 9]  (a) {\large $v_9$};
			
			\path[draw,thick]
			(1) edge node {} (2)
			(2) edge node {} (3)
			(2) edge node {} (4)
			(4) edge node {} (5)
			(4) edge node {} (6)
			(6) edge node {} (7)
			(6) edge node {} (8)
            (8) edge node {} (9);
		\end{scope}
		\begin{scope}[xshift=15cm]
			
			\node[source] (2) {$1$};
            \node[draw = none, fill = none, above=1mm of 2]  (a) {\large $v_3$};
			\node[non-source] (1) [above left = 0.75cm of 2] {$2$};
            \node[draw = none, fill = none, above=1mm of 1]  (a) {\large $v_1$};
			\node[non-source] (3) [below left = 0.75cm of 2] {$2$};
            \node[draw = none, fill = none, below=1mm of 3]  (a) {\large $v_2$};
			\node[non-source] (4) [right = 0.5cm of 2] {$2$};
            \node[draw = none, fill = none, above=1mm of 4]  (a) {\large $v_4$};
			\node[non-source] (5) [below = 0.4cm of 4] {$3$};
            \node[draw = none, fill = none, below=1mm of 5]  (a) {\large $v_5$};
            \node[draw = none, fill = none, below=8mm of 5]  (a) {\large round $3$};
			\node[source] (6) [right = 0.5cm of 4] {$2$}; 
            \node[draw = none, fill = none, above=1mm of 6]  (a) {\large $v_6$};
			\node[non-source] (7) [below = 0.4cm of 6] {$3$};
            \node[draw = none, fill = none, below=1mm of 7]  (a) {\large $v_7$};
			\node[non-source] (8) [right = 0.5cm of 7] {$3$};
            \node[draw = none, fill = none, below=1mm of 8]  (a) {\large $v_8$};
            \node[source] (9) [above = 0.4cm of 8] {$3$};
            \node[draw = none, fill = none, above=1mm of 9]  (a) {\large $v_9$};
			
			\path[draw,thick]
			(1) edge node {} (2)
			(2) edge node {} (3)
			(2) edge node {} (4)
			(4) edge node {} (5)
			(4) edge node {} (6)
			(6) edge node {} (7)
			(6) edge node {} (8)
            (8) edge node {} (9);
		\end{scope}
	\end{tikzpicture}
	}%
	\caption{
    \rev{Suppose} we burn the tree on 9 vertices shown above using the sequence~$(v_3,v_6,v_9)$. 
    The four copies show the tree at the end of rounds~0--3. 
    Labels indicate the round in which each vertex burns, unlabelled vertices are unburned in that round, and sources are shown in gray. 
    In round~1, we burn~$v_3$; in round~2, its neighbours $v_1,v_2,v_4$ burn and we burn~$v_6$; in round~3, the remaining neighbours $v_5,v_7,v_8$ burn and we burn~$v_9$. 
    Thus all vertices burn by round~3. 
    Since the maximum degree is~3, at most five vertices can be burned by round~2, so no sequence of length~2 exists, and the burning number of the tree is~3.
 }
	\label{fig:ExampleBurning}
\end{figure}

\section{Burning graph powers}\label{sec:GraphPowers}

In this section, we show that given a graph~$G$ on~$n$ vertices and \SK{an} integer~$2\le k\le \diam(G)-1$,
\begin{equation}\label{eqn:GraphPowerIntro}
    b(G^k) \le \left\lceil\sqrt{\frac{4(k-1)}{k^2}n}~\right\rceil.
\end{equation}
We do this in three steps.
Firstly, in \Cref{subsec:TreePowers}, we show that for a tree~$T$ with diameter~$d$,~$T^k$ contains a~$(k+1)^+$-branching spanning tree whenever~$1\le k\le d-1$ (\Cref{lem:GraphPowerHask-angular}).
Next, in \Cref{subsec:BBranching}, we show that every~$(k+1)^+$-branching tree on~$n$ vertices has burning number at most $\left\lceil\sqrt{\frac{4(k-1)}{k^2}n}~\right\rceil$ (\Cref{thm:k-angBoundle}).
Finally, in \Cref{subsec:BGraphPowers}, we combine these results. 
Every graph~$G$ contains a spanning tree~$T$. Noting that~$T^k$ is a subgraph of~$G^k$, we argue, using the steps above, that~$G^k$ must contain a~$(k+1)^+$-branching spanning tree, which establishes Inequality~\ref{eqn:GraphPowerIntro} (\Cref{thm:GraphPowerBN}). 

\subsection{Tree powers and \texorpdfstring{$k^+$}{k+}-branching spanning trees}\label{subsec:TreePowers}

We recall the following key lemma, which essentially states that every tree contains a bridge whose removal separates the tree into two subtrees in such a way that one of the subtrees will be `large’ yet consist only of `small' subsubtrees directly attached to the bridge.

\begin{lemma}[Lemma 8 of \cite{ning2025burning}]\label{lem:TheRightBridgeStrong}
    Let~$T$ be a tree on~$n$ vertices, where~$n\ge 3$. 
    Then for any real number~$p\in[1,n-1)$, there exists a vertex~$v$ in~$T$ with~$N_1(v) = \{v_1,v_2,\ldots, v_m\}$ where~$m\ge2$, such that~$|T_v(vv_m)|>p$ and $|T_{v_i}(vv_i)|\le p$ for all~$i\in[m-1].$
\end{lemma}

The following lemma \NV{establishes} a connection between tree powers and~$k^+$-branching spanning trees.

\begin{lemma}\label{lem:GraphPowerHask-angular}
    Let~$T = (V,E)$ be a tree with diameter~$d$, and let~$1\le k\le d-1$ be \SK{an} integer.
    Then~$T^k$ contains a~$(k+1)^+$-branching spanning tree.
\end{lemma}

\begin{proof}
    We first show that the claim holds when~$\rad(T)\le k$.
    Then there exists a vertex~$u\in V$ such that~$d(u,v)\le k$ for all~$v\in V$, which means~$uv\in E(T^k)$ for all~$v\in V$.
    Then the star graph with~$u$ as the central vertex is a spanning tree of~$T^k$, and we are done.
    Note that this star graph is~$(k+1)^{+}$-branching since the degree of~$u$ in~$T^k$ is $|V| - 1 \ge d \ge k+1$.

    Now suppose $\rad(T)>k$.
    Let~$T_0:= T$.
    For every integer~$i\ge 1$, we \SK{recursively} construct a tree~$T_{i}$ from~$T_{i-1}$, until we obtain a tree~$T_\ell$ for some integer~$\ell$, such that~$\rad(T_\ell)\le k$.
    \rev{
    In every step, we find a vertex $x_{i}$ with a distinguished neighbour $v_{m_{i}}^{i}$ such that two conditions hold: firstly, all vertices closer to $x_{i}$ than to $v_{m_{i}}^{i}$ are neighbours of $x_{i}$ in $T^{k}$ and secondly, there are at least $k$ such vertices. This then allows us to construct a spanning tree of $T^{k}$ by making $x_{i}$ an internal vertex and attaching those vertices closer to $x_{i}$ than to $v_{m_{i}}^{i}$ as leaves to $x_{i}$.
    \YM{We refer the reader to \Cref{fig:GraphPowerHask-angular} for an illustration of the construction.}}

    \begin{figure}[ht]
\centering
\resizebox{\textwidth}{!}{%
\begin{tikzpicture}[scale=1.1, every node/.style={font=\small}]


\node at (0.7,3.1) {$T_{i-1}$};

\node[circle, fill, inner sep=1.5pt] (x1) at (0.7,1.6) {};
\node[below] at (x1) {$x_i$};

\node[circle, fill, inner sep=1.5pt] (v1) at (0,2.2) {};
\node[right=3pt] at (v1) {$v^i_1$};

\node[circle, fill, inner sep=1.5pt] (v2) at (0,1.6) {};
\node[above=1pt] at (v2) {$v^i_2$};


\node[circle, fill, inner sep=1.5pt] (vm1) at (0,0.8) {};
\node[right=3pt] at (vm1) {$v^i_{m_i-1}$};


\draw (x1) -- (v1);
\draw (x1) -- (v2);
\draw[densely dotted] ($(v2) + (0,-0.2)$) -- ($(vm1) + (0,0.2)$);
\draw (x1) -- (vm1);

\foreach \n in {v1,v2,vm1} {
    \draw (\n) -- ++(-0.6,0.2);
    \draw (\n) -- ++(-0.6,0);
    \draw (\n) -- ++(-0.6,-0.2);
    \draw[densely dotted] ($( \n ) + (-0.5,0)$) -- ++(0,-0.166);
}

\node[circle, fill, inner sep=1.5pt] (y1) at (1.4,1.6) {};
\node[below] at (y1) {$v^i_{m_i}$};
\draw (x1) -- (y1);
\foreach \dy in {0.4,0,-0.4} {
    \draw (y1) -- ++(0.5,\dy);
    \draw[densely dotted] ($(y1) + (0.4,0)$) -- ++(0,-0.3);
}

\draw[->, thick] (2.4,1.6) -- ++(0.6,0);

\node at (4.2,3.1) {$T_i$};

\node[circle, fill, inner sep=1.5pt] (x2) at (3.5,1.6) {};
\node[below] at (x2) {$x_i$};

\node[circle, fill, inner sep=1.5pt] (y2) at (4.2,1.6) {};
\node[below] at (y2) {$v^i_{m_i}$};

\draw (x2) -- (y2);

\foreach \dy in {0.4,0,-0.4} {
    \draw (y2) -- ++(0.8,\dy);
    \draw[densely dotted] ($(y2) + (0.6,0)$) -- ++(0,-0.3);
}

\node[draw, rectangle, dashed, rounded corners=5pt, minimum width=2.1cm, minimum height=2.6cm] (box1) at (0.08,1.3) {};
\node[] at ([xshift=0.13cm, yshift=-1cm] box1.east) {$\dagger$};

\node[] at (3.313,0.6) {$\dagger = {T_{i-1}}_{x_i}(x_iv^i_{m_i})$};

\node[draw, rectangle, dashed, rounded corners=5pt, minimum width=1cm, minimum height=0.6cm] (box2) at (-0.3,0.8) {};
\node[] at ([xshift=0cm, yshift=-0.15cm] box2.south) {$\star$};

\node[] at (3.7,0) {$\star = {T_{i-1}}_{v^i_{m_i-1}}(x_iv^i_{m_i-1})$};

\node[draw, rectangle, rounded corners=5pt, minimum width=7cm, minimum height=4cm] (box3) at (2.1,1.55) {};

\node[draw, rectangle, rounded corners=5pt, minimum width=3.35cm, minimum height=3.2cm] (box4) at (0.55,1.4) {};

\node[draw, rectangle, rounded corners=5pt, minimum width=2.05cm, minimum height=1.5cm] (box5) at (4.2,1.55) {};

\begin{scope}[xshift = 6.5cm, yshift=3cm]


\node at (3.5,0.1) {$S_i$};

\node[circle, fill, inner sep=1.5pt] (xs2) at (3,-1.4) {};
\node[below] at (xs2) {$x_i$};
\node[circle, draw, inner sep=1.5pt] (p2) at (3.3,-1.4) {};
\draw (xs2) -- (p2);
\draw (p2) -- ++(0.3,0);
\node at ($(p2) + (0.55,0)$) {$\cdots$};

\draw[->, thick] (2.53,-1.4) -- ++(-0.6,0);

    \node at (0.7,0.1) {$S_{i-1}$};
    
    \node[circle, draw, inner sep=1.5pt] (xs1) at (0.7,-1.4) {};
    \node[above] at (xs1) {$x_i$};
    \node[circle, draw, inner sep=1.5pt] (p1) at (1,-1.4) {};
    \draw (xs1) -- (p1);

    \draw (p1) -- ++(0.3,0);
    \node at ($(p1) + (0.55,0)$) {$\dots$};
    
    

    \node[circle, fill, inner sep=1.5pt] (sv1) at (0,-0.8) {};
    \node[right=3pt] at (sv1) {$v^i_1$};
    
    \node[circle, fill, inner sep=1.5pt] (sv2) at (0,-1.4) {};
    \node[left] at (sv2) {$v^i_{m_i}$};
    
    \node[circle, fill, inner sep=1.5pt] (sv) at (0,-2) {};
    \node[below=3pt] at (sv) {};

    \draw (xs1) -- (sv1);
    \draw (xs1) -- (sv2);
    \draw (xs1) -- (sv);
    \draw[densely dotted] ($(sv1) + (0,-0.2)$) -- ($(sv2) + (0,0.2)$);
    \draw[densely dotted] ($(sv2) + (0,-0.2)$) -- ($(sv) + (0,0.2)$);

    \node[draw, rectangle, rounded corners=5pt, minimum width=6cm, minimum height=4cm] (box1) at (1.75,-1.45) {};

    \node[draw, rectangle, rounded corners=5pt, minimum width=2.7cm, minimum height=2cm] (box2) at (0.55,-1.4) {};

    \node[draw, rectangle, rounded corners=5pt, minimum width=1.8cm, minimum height=2cm] (box3) at (3.5,-1.4) {};

\end{scope}

\end{tikzpicture}
}
\caption{A visualisation of the proof of \Cref{lem:GraphPowerHask-angular}.
Left: Construction of~$T_i$ from~$T_{i-1}$ when~$\rad(T_{i-1}) > k$.
Right: Construction of~$S_{i-1}$ from~$S_{i}$. 
In both figures, the dotted lines indicate the \rev{(potential)} existence of more vertices.
In the right diagram, filled (black) vertices are leaves and unfilled vertices are those of degree at least~$k$.
Furthermore, the visible leaves in~$S_{i-1}$ are all vertices from~$V(T_{i-1})\setminus V(T_i)$.}
\label{fig:GraphPowerHask-angular}
\end{figure}


    \rev{We now formally describe the construction.}
    Suppose we have~$T_{i-1}$ such that~$\rad(T_{i-1}) >k$.
    As~$k \ge 1$, we know that~$T_{i-1}$ has at least~$3$ vertices. 
    Hence, we may apply \Cref{lem:TheRightBridgeStrong} to~$T_{i-1}$ for the fixed value~$k$ to obtain a vertex~$x_i$ in~$T_{i-1}$ with~$N^{T_{i-1}}_1(x_i)
    = \{v^i_1,v^i_2, \ldots, v^i_{m_i}\}$ where~$m_i\ge 2$, such that
    \begin{align*}
        |{T_{i-1}}_{x_i}(x_iv^i_{m_i})|>k && \text{and} && |{T_{i-1}}_{v^i_j}(x_iv^i_j)|\le k && \text{for all} && j\in[m_i-1].
    \end{align*}
    Consequently, for every~$y\in V({T_{i-1}}_{x_i}(x_iv^i_{m_i}))\setminus \{x_i\}$, we have~$d_{T_{i-1}}(x_i,y) \le k$, and therefore~$x_iy\in E(T_{i-1}^k)$.
    Now remove all vertices in~$V({T_{i-1}}_{x_i}(x_iv^i_{m_i}))\setminus \{x_i\}$ from~$T_{i-1}$ to obtain a tree~$T_i$.

    Since~$T$ is a finite graph, we obtain a tree~$T_\ell$ for~$\ell\in \mathbb{Z}$ where~$\rad(T_\ell)\le k$. 
    Let~$\ell$ be the smallest such integer.
    We split into two cases depending on the value of~$\diam(T_\ell)$.
    \begin{itemize}
        \item Suppose that $\diam(T_\ell)<k+1$.
        We claim that this implies~$\rad(T_{\ell-1})\le k$, contradicting the choice of~$\ell$.
        Since~$x_{\ell}$ is a leaf in~$T_\ell$, for every vertex~$y\in T_\ell$, we have
        \[d_{T_{\ell-1}}(x_{\ell},y) = d_{T_\ell}(x_{\ell},y) \le k,\]
        where the inequality follows from the assumption that $\diam(T_\ell)<k+1$.
        Moreover, by construction, all remaining vertices~$z$ of $T_{\ell-1}$ must also satisfy the inequality~$d_{T_{\ell-1}}(x_{\ell},z) \le k$, since these are precisely the vertices deleted to obtain~$T_\ell$.
        Consequently, $\rad(T_{\ell-1})\le k$, a contradiction.
        Therefore, this case cannot occur.
        \item Suppose that $\diam(T_\ell)\ge k+1$.
        Then~$T_\ell$ satisfies the conditions of \NV{this} lemma. 
        Since~$\rad(T_\ell)\le k$, we have from the arguments used at the start of the proof, that~$T_\ell^k$ contains a spanning tree~$S_\ell$ that is~$(k+1)^+$-branching (this is in particular a star graph). 
        For each~$i= \ell, \ell-1, \ldots, 1$, apply the following (\YM{see \Cref{fig:GraphPowerHask-angular} for an illustration of the construction}):
        \begin{itemize}
            \item add vertices from~$V(T_{i-1})\setminus V(T_{i})$, and add an edge from~$x_i$ to each of these vertices;
            \item call the resulting tree~$S_{i-1}$.
        \end{itemize}
        We claim that~$S_{i-1}$ is a spanning tree of~$T_{i-1}^k$ that is~$(k+1)^+$-branching.
        But this is immediate by choice of~$x_i$.
        In particular,
        \[|V(T_{i-1})\setminus V(T_{i})| = |{T_{i-1}}_{x_i}(x_iv^i_{m_i})| - 1 \ge k,\]
        and as~$x_i$ is a leaf in tree~$T_i$, it must be of degree at least~$k+1$ in~$T_{i-1}$.
        Thus, we obtain a spanning tree~$S_0$ of~$T^k_0 = T^k$ that is~$(k+1)^+$-branching, which is what we wanted.\qedhere
    \end{itemize}
\end{proof}

We show that \Cref{lem:GraphPowerHask-angular} is tight in the following manner.

\begin{observation}
    For any \rev{positive} integer~$k$, 
    the \rev{path} graph \rev{power}~$P_{2k+2}^k$ does not contain a~$(k+2)^+$-branching spanning tree.
\end{observation}
\begin{proof}
    Since~$\rad(P_{2k+2}) = k+1$, any spanning tree of~$P_{2k+2}^k$ must contain at least 2 internal vertices.
    If~$P_{2k+2}^k$ contains a~$(k+2)^+$-branching spanning tree, then it must contain at least~$2(k+2) = 2k+4 > 2k+2$ vertices, which is not possible.
\end{proof}

\subsection{Burning number bound of \texorpdfstring{$k^+$}{k+}-branching trees}\label{subsec:BBranching}

The following is a generalisation of Lemma 3 in \cite{murakami2024burning}, which states that we can bound the number of internal vertices depending on the total number of vertices.

\begin{lemma}\label{lem:k-angInt}
    Let~$k \geq 2$ and~$I \geq 0$ be integers and let~$T$ be a~$k^{+}$-branching tree.
    If~$|T| \leq (k-1)(I+1)+1$ then~$T$ contains at most~$I$ internal vertices.
    Equivalently, if~$T$ contains at least~$I$ internal vertices, then~$|T|\ge I(k-1)+2$.
\end{lemma}
\begin{proof}
    \SK{Suppose that $T$ contains at least $I$ internal vertices.
    Then $T$ has at most $|T|-I$ leaves. Furthermore, since $T$ is $k^{+}$-branching, we obtain
    \[\sum_{v \in V(T)} \deg(v) \ge 1 \cdot (|T|-I) + k \cdot I = |T| + I(k-1).\]
    Moreover, by the handshaking lemma, we have $\sum_{v \in V(T)} \deg(v) = 2(|T|-1)$.
    Hence, the above inequality yields $2(|T|-1) \ge |T| + I(k-1)$, implying~$|T|\ge I(k-1)+2$, as desired.}
\end{proof}

In the proof of the following results, we require a notion of graph burning 
where multiple sources can be burned in round 1. 
Let~$G$ be a graph.
For~$U\subseteq V(G)$ and 
vertices~$x_i\in V(G)$, let~$M = (U\cup\{x_1\}, x_2, \ldots, x_k)$ be a 
sequence. In round 1, burn all vertices in the set~$U\cup\{x_1\}$; in round 
$i$ for~$i\ge 2$, proceed as in the traditional burning \NV{process}. We call~$M$ a 
\emph{modified burning sequence} for~$G$ if all vertices of~$G$ are burned 
after round~$k$. The \emph{modified burning number}~$b^{U}(G)$ of~$G$ is the 
length of a shortest modified burning sequence for~$G$, with some 
set~$U\subseteq V(G)$.

The following is a generalisation of Lemma~4 in \cite{murakami2024burning}.
Let~$yz$ be an edge of a graph~$G$. 
By \emph{contracting}~$yz$, we mean deleting~$y$ and adding an edge~$zx$ for every~$x\in N^G_1(y)\setminus\{z\}$.

\begin{lemma}\label{lem:HITwithDk-1Burned}
	Let~$T$ be a~$(k-1)^+$-branching tree with \SK{exactly} one degree-$(k-1)$ vertex~$y$, and suppose~$y$ has a non-leaf neighbour~$z$.
    Let~$T'$ be the~$k^+$-branching tree obtained from~$T$ by contracting~$yz$. 
    Then
	\[b^{\{y\}}(T) \le b(T').\]
\end{lemma}
\begin{proof}
	Let~$(x_1,\ldots, x_k)$ be a (not necessarily optimal) burning sequence  
	for~$T'$. We claim that $(\{y,x_1\}, x_2,\ldots, x_k)$ is a modified 
	burning sequence for~$T$. This would imply that any burning sequence 
	for~$T'$ yields a modified burning sequence of the same length for~$T$, 
	from which we may conclude that $b^{\{y\}}(T) \le b(T')$.
	
	Suppose for a contradiction that $(\{y,x_1\}, x_2,\ldots, x_k)$ is not a 
	modified burning sequence for~$T$. Then there exists a vertex~$w$ 
	in~$T$ that is not burned at the end of round~$k$. Clearly~$w$ cannot be 
	one of the sources. Since~$(x_1,\ldots, x_k)$ is a burning sequence 
	for~$T'$,~$w$ is a burned vertex at the end of round~$k$. Suppose that~$w$ 
	becomes burned in~$T'$ in round~$j$ for some~$j\le k$. Since~$w$ is not a 
	source, there must be a source~$x_i$ with~$i<j$ such that~$d_{T'}(w,x_i) 
	=j-i$.
	
	Because~$w$ remains unburned in~$T$, we must have that $d_{T}(w,x_i) 
	>j-i$. Then the path from~$w$ to~$x_i$ in~$T$ must contain the 
	vertex~$y$, as this is the only difference between trees~$T$ and~$T'$. It 
	follows that~$d_{T}(w,x_i) = j-i+1$. But then~$d_{T}(w,y)\le j-i$. This 
	would mean that~$w$ becomes a burned vertex in~$T$ no later than 
	round~$1+j-i$, since~$y$ is burned in round 1. Since~$1+j-i\le j\le k$, this 
	means that~$w$ is burned in~$T$ at the end of round~$k$. This gives the 
	required contradiction.
\end{proof}

We now give a burning number bound for $k^+$-branching trees on~$n$ vertices.
This is an extension of Theorem~1 of~\cite{murakami2024burning}, which corresponds to the case~$k=3$.

\begin{theorem}\label{thm:k-angBoundle}
    Let~$k\ge3$ be an integer and let~$T$ be a~$k^+$-branching tree on~$n$ vertices.
    Then $b(T) \le \BoundC$, 
    where $\displaystyle c=\frac{4(k-2)}{(k-1)^2}$.
\end{theorem}
\begin{proof}
    Let~$I$ be the number of internal vertices in~$T$.
    We split into two cases based on the \SK{value} of~$n$.
    \medskip

    \noindent\textbf{Case 1: $n\le 4(k-2)+1$.}
    First observe that
    \begin{equation}\label{eqn:cn}
        \BoundC = \left\lceil{\sqrt{\frac{4(k-2)n}{(k-1)^2}}}~\right\rceil \ge \left\lceil \frac{n}{k-1}\right\rceil.
    \end{equation}
    \Cref{lem:k-angInt} gives 
    \begin{equation}\label{eqn:nI}
        n \ge I(k-1)+2.
    \end{equation}
    Plugging \Cref{eqn:nI} into \Cref{eqn:cn} gives
    \[\BoundC\ge \left\lceil I + 2/(k-1)\right\rceil = I + 1.\]
    Arbitrarily order the internal vertices of~$T$, using~$(v_1,\ldots, v_I)$.
    As every leaf is \NV{at} distance \NV{one} from \SK{an} internal vertex, it follows that~$(v_1,\ldots, v_I, v)$ is a valid burning sequence for~$T$, for any~$v\in V(T)$.
    It follows then that,
    \[b(T) \le I+1 \le \BoundC,\]
    which is what we wanted to show.
    \medskip

    \noindent\textbf{Case 2: $n\ge4(k-2)+2$.}
    Clearly,~$I\ge1$.
    We \NV{use} induction on~$I$.
    For the base case, we consider the case~$I=1$.
    This means that~$T$ is a star graph, and thus~$b(T) = 2$.
    Since \SK{$k \ge 3$, and by taking~$n>4(k-2)$, we have}
    \begin{equation}\label{eqn:cn>1}
        \sqrt{cn}> \frac{4(k-2)}{(k-1)} \ge \SK{2}.
    \end{equation}
    \SK{Consequently,} the claim holds.

    We now assume that for all $k^+$-branching trees with at most $I-1$ internal vertices, the theorem holds.
    We show that the claim holds for~$T$, which is a~$k^+$-branching tree with~$I$ internal vertices.
    By \Cref{lem:TheRightBridgeStrong},~$T$ contains a vertex~$x$ with neighbours~$v_1,\ldots, v_m = y$ such that for $i\in[m-1]$,
    \begin{align*}
        |T_{v_i}(xv_i)|\le (k-1)(\BoundC-1) && \text{and} && |T_x(xy)|> (k-1)(\BoundC-1).
    \end{align*}
    \YM{To show that such a vertex exists, we must verify that $(k-1)(\BoundC-1)\in [1,n-1)$ to satisfy the conditions of \Cref{lem:TheRightBridgeStrong}.}
    \begin{itemize}
        \item Observe that $(k-1)(\BoundC-1)\ge 1$,
        as $k-1\ge 2$ and $\BoundC-1 \ge \SK{2}$ from \Cref{eqn:cn>1}.
        \item \YM{We now wish to show that~$(k-1)(\BoundC-1)<n-1$.
        As~$\lceil x\rceil < x+1$ for all~$x\in\mathbb{R}$, we have
        \[(k-1)(\BoundC-1) < (k-1)\sqrt{cn}.\]
        So the original claim is true if
        \[(k-1)\sqrt{cn} = \sqrt{4(k-2)n} \le n-1.\]
        As both sides are non-negative, we can square them to show that this inequality is true if the following is true.
        \[4(k-2)n \le n^2-2n+1.\]
        Simplifying, we wish to show that
        \[n^2-(4k-6)n+1 \ge0.\]
        But as~$n\ge 4(k-2)+2$ by assumption, this must be true since
        \[n^2-(4k-6)n+1 = n^2-(4(k-2)+2)n + 1 \ge n^2 - n^2 + 1 = 1 \ge 0.\]
        }
    \end{itemize}
    Hence,~$(k-1)(\BoundC-1)\in [1,n-1)$, and thus our invocation of \Cref{lem:TheRightBridgeStrong} is valid.
    We claim that the distance from $x$ to any vertex 
	in~$H_i:=V(T_{v_i}(xv_i))$ is at most~$\BoundC-1$.
    \YM{See \Cref{fig:k-angBoundle} for an illustration.} 

    Consider the induced subtree~$T_i:=T[H_i\cup \{x\}]$. 
	Since~$|H_i| \le(k-1)(\BoundC-1)$, it follows that~$|T_i| \le (k-1)(\BoundC-1) + 1$. 
    By \Cref{lem:k-angInt},~$T_i$ contains at most~$\BoundC-2$ internal vertices. 
    It follows immediately that any 
	path in~$T$ from~$x$ to a vertex in~$H_i$ is of \NV{length} at 
	most~$\BoundC-1$. 
    As this is true for all~$i\in[m-1]$, it follows that the distance from~$x$ to every vertex in~$T_x(xy)$ is at most $\BoundC-1$.

    We now claim that~$T$ can be burned in at most~$\BoundC$ rounds, by 
	burning~$x$ in round 1. 
    Observe that upon burning~$x$ in round 1, with no 
	additional sources in~$T_x(xy)$, all vertices of~$T_x(xy)$ will be burned 
	by the end of round $\BoundC$. Indeed, this occurs since the distance from~$x$ to every vertex in $T_x(xy)$ is at most $\BoundC-1$.

    The vertex~$y$ will become burned in round 2, as it is a neighbour of~$x$.
    If~$y$ is a leaf, then we are done, so suppose~$y$ is not a leaf.
	It remains to show that~$b^{\{y\}}(T_y(xy))\le \BoundC-1$. 
    We split into two subcases, depending on the degree of~$y$ in~$T_y(xy)$.
    \medskip
    
    \noindent\textbf{Subcase 1: $y$ is a degree-$(k-1)$ vertex in~$T_y(xy)$.}
    If~$y$ only has leaf neighbours, then we are done, as~$\BoundC\ge 3$ from \Cref{eqn:cn>1}. 
    So suppose~$y$ has a neighbour~$z$ that is an internal vertex. 
    \rev{Since all internal vertices in~$T_y(xy)$, excluding~$y$, are of degree at least~$k$ (since~$T$ is~$k^+$-\rev{branching}), the tree~$T_y(xy)$ is a~$(k-1)^+$-branching tree with~$y$ as its unique vertex of degree $k-1$.
    So we may apply \Cref{lem:HITwithDk-1Burned}.}
    We contract the edge~$yz$ to obtain a~$k^+$-branching tree~$T'$.
    By \Cref{lem:HITwithDk-1Burned}, we have~$b^{\{y\}}(T_y(xy))\le b(T')$. 
    \medskip
    
    \noindent\textbf{Subcase 2: $y$ is not a degree-$(k-1)$ vertex in~$T_y(xy)$.}
    Then~$T_y(xy)$ itself is a $k^+$-branching tree and we 
	have~$b^{\{y\}}(T_y(xy))\le b(T_y(xy))$. 
    \medskip
    
    We now combine the analysis for the two subcases. 
    Note that
	\[|T'| < |T_y(xy)| = n - |T_x(xy)| \le n-(k-1)(\BoundC-1)-1= n- (k-1)\BoundC + k-2.\]

    Continuing the calculation, we obtain
    \begin{align*}
        c|T_y(xy)| &\le cn - c(k-1)\BoundC + c(k-2)\\
        &\le \BoundC^2 - c(k-1)\BoundC + c(k-2)\\
        &= \left(\BoundC - c(k-1)/2\right)^2 - c^2(k-1)^2/4 + c(k-2)\\
        &= \left(\BoundC - c(k-1)/2\right)^2
    \end{align*}
    where the final equality follows \NV{from the} definition of~$c$. Finally, this gives
    \begin{align*}
        \left\lceil\sqrt{c|T_y(xy)|}\right\rceil &\le \left\lceil \BoundC - c(k-1)/2 \right\rceil\\
        &= \left\lceil \BoundC - 2(k-2)/(k-1)\right\rceil\\
        &= \BoundC - 1.
    \end{align*}
    Since~$x$ is an internal vertex,~$T'$ and~$T_y(xy)$ both contain fewer internal vertices than~$T$.
	It follows by \NV{the} induction hypothesis that~$b(T') \le \BoundC -1$ 
	and~$b(T_y(xy))\le \BoundC -1$. Therefore, we obtain $b^{\{y\}}(T_y(xy))\le 
	\BoundC -1$, and we are done.
\end{proof}

\begin{figure}[ht]
\centering
\resizebox{0.7\textwidth}{!}{%
\begin{tikzpicture}[scale=1.1, every node/.style={font=\small}]
    

\node[circle, fill, inner sep=1.5pt] (x1) at (3,3) {};
\node[below] at (x1) {$x$};

\node[circle, fill, inner sep=1.5pt] (v1) at (0,4) {};
\node[above=1pt] at (v1) {$v_1$};

\node[circle, fill, inner sep=1.5pt] (v2) at (0,3) {};
\node[above=1pt] at (v2) {$v_2$};


\node[circle, fill, inner sep=1.5pt] (vm1) at (0,2) {};
\node[below=3pt] at (vm1) {$v_{m-1}$};

\draw (x1) -- (v1);
\draw (x1) -- (v2);
\draw[densely dotted] ($(v2) + (0,-0.2)$) -- ($(vm1) + (0,0.2)$);
\draw (x1) -- (vm1);

\foreach \n in {v1,v2,vm1} {
    \draw (\n) -- ++(-0.6,0.2);
    \draw (\n) -- ++(-0.6,0);
    \draw (\n) -- ++(-0.6,-0.2);
    \draw[densely dotted] ($( \n ) + (-0.5,0)$) -- ++(0,-0.166);
}

\node[circle, fill, inner sep=1.5pt] (y1) at (4,3) {};
\node[below] at (y1) {$y$};
\draw (x1) -- (y1);
\foreach \dy in {0.4,0,-0.4} {
    \draw (y1) -- ++(0.8,\dy);
    \draw[densely dotted] ($(y1) + (0.6,0)$) -- ++(0,-0.3);
}

\node[draw, dashed, rectangle, rounded corners=5pt, minimum width=1.5cm, minimum height=0.9cm] (box1) at (-0.2,4) {};
\node[anchor=west, font = \tiny] at ([xshift=1pt]box1.east) {$T_{v_1}(xv_1)$};

\node[draw, rectangle, rounded corners=5pt, minimum width=4.8cm, minimum height=1.1cm] (box2) at (1.2,3) {};
\node[anchor=south, font = \tiny] at ([xshift=-5pt]box2.north east) {$T_2$};

\node[draw, dashed, rectangle, rounded corners=5pt, minimum width=1.5cm, minimum height=0.9cm] (box) at (-0.2,3) {};
\node[anchor=west, font=\tiny] at ([xshift=1pt,yshift=5pt]box.east) {$T_{v_2}(xv_2)$};

\node[draw, dashed, rectangle, rounded corners=5pt, minimum width=1.5cm, minimum height=0.9cm] (box3) at (-0.2,2) {};
\node[anchor=west, font=\tiny] at ([xshift=1pt]box3.east) {$T_{v_{m-1}}(xv_{m-1})$};

\node[draw, rectangle, rounded corners=5pt, minimum width=5.2cm, minimum height=3.5cm] (box0) at (1.15,3) {};
\node[anchor=west, font=\footnotesize] at ([xshift=2pt,yshift=-5pt]box0.north east) {$T_x(xy)$};
\end{tikzpicture}
}
\caption{A visualisation of the proof of \Cref{thm:k-angBoundle}.
We indicate subtrees with rectangular boxes.
In particular, dashed boxes are used to indicate the trees~$T_{v_1}(xv_1), T_{v_2}(xv_2)$, and~$T_{v_{m-1}}(xv_{m-1})$; solid boxes are used to indicate the trees~$T_2$ and~$T_x(xy)$.
$T_2$ contains at most~$\BoundC -2$ internal vertices.
All vertices in~$T_x(xy)$ are distance at most~$\BoundC-1$ away from~$x$.
Therefore, selecting~$x$ as a first source burns all vertices of~$T_x(xy)$ in at most~$\BoundC$ rounds, without selecting additional sources.
}
\label{fig:k-angBoundle}
\end{figure}

\subsection{Burning number bound of graph powers}\label{subsec:BGraphPowers}

We recall the following lemma on spanning subtrees.

\begin{lemma}[Corollary 6 of \cite{bonato2016burn}]\label{lem:SpanningTree}
    For a graph~$G$, we have
    \[b(G) = \min\{b(T): T \text{ is a spanning subtree of } G\}.\]
\end{lemma}

We are now ready to prove the main theorem.

\begin{theorem}\label{thm:GraphPowerBN}
    Let~$G$ be a graph on~$n$ vertices, and let~$k$ be \NV{an} integer where~$2\le k\le \diam(G)$.
    Then \[b(G^k)\le \left\lceil\sqrt{\frac{4(k-1)}{k^2}n}~\right\rceil.\]
\end{theorem}
\begin{proof}
    Suppose first that~$k=\diam(G)$.
    Then~$G^k$ is isomorphic to the complete graph~$K_n$, and thus~$b(G^k) = 2$.
    On the other hand, we have
    \[\frac{4(k-1)n}{k^2}\ge \frac{4(k-1)(k+1)}{k^2}\ge3,\]
    where the first inequality follows from~$n\ge \diam(G) + 1$, and the second inequality follows as $\frac{k^2-1}{k^2}$ is monotonically increasing for~$k\ge 2$ with value~$3/4$ when~$k=2$.
    Then we have
    \[b(G^k) = 2 \le \left\lceil\sqrt{\frac{4(k-1)}{k^2}n}~\right\rceil,\]
    as required.

    Now suppose that~$2\le k <\diam(G)$.
    Take any spanning tree~$T$ of~$G$. 
    By \Cref{lem:GraphPowerHask-angular},~$T^k$ contains a spanning tree $S$ that is~$(k+1)^+$-branching.
    Since~$T^k$ is a subgraph of~$G^k$,~$S$ is also a spanning tree of~$G^k$.
    Then we have
    \[b(G^k) \le b(S) \le  \left\lceil\sqrt{\frac{4(k-1)}{k^2}n}~\right\rceil,\]
    where the first and second inequalities follow from~\Cref{lem:SpanningTree} and~\Cref{thm:k-angBoundle}, respectively.
\end{proof}

We highlight the~$k=2$ case of \Cref{thm:GraphPowerBN}.

\begin{corollary}\label{cor:GraphSquareBN}
    Let~$G$ be a graph on~$n$ vertices.
    Then \(b(G^2)\le \left\lceil\sqrt{n}~\right\rceil.\)
\end{corollary}

\section{Larger number of vertices}

In this section, we examine how the burning number bound for~$k^+$-branching trees of \Cref{thm:k-angBoundle} compares to the state of the art from the graph burning literature.
We first recall these results.

\begin{theorem}[Theorem 1 of \cite{bastide2023improved}]\label{thm:Bastide}
    If~$G$ is a connected graph on~$n$ vertices, then
    \[b(G) \le \left\lceil \sqrt{4n/3} \right\rceil + 1.\]
\end{theorem}

\begin{theorem}[Theorem 1.2 of \cite{norin2024burning}]\label{thm:Norin}
    If~$G$ is a connected graph on~$n$ vertices, then
    \[b(G) \le (1+o(1))\sqrt{n}.\]
\end{theorem}

To apply \Cref{thm:Bastide,thm:Norin}, we use an argument based on dominating sets.
A set~$D$ of vertices is a \emph{dominating set} for a graph~$G$ if every vertex in~$G$ is either contained in~$D$ or has a neighbour in~$D$.
We first show that a~$k^+$-branching tree must contain at most~$n/(k-1)$ internal vertices.
These vertices form a connected dominating set (i.e., the internal vertices induce a connected graph), and upon burning all vertices therein, we require at most one additional round to burn the leaves in the tree.

\begin{lemma}\label{lem:k-branchleaflowerbound}
    Let~$T$ be a~$k^+$-branching tree on~$n$ vertices, where~$k\ge3$. Then~$T$ contains at least~$\displaystyle \left\lceil \frac{n(k-2)}{(k-1)}\right\rceil$ leaves.
\end{lemma}
\begin{proof}
    Let~$I,\ell$ be the number of internal vertices and leaves of~$T$, respectively.
    By the handshaking lemma, we obtain
    $2(n-1)\ge kI + \ell$,
    for which substituting~$I = n-\ell$ gives
    \[\ell\ge \frac{n(k-2)}{(k-1)} + \frac{2}{(k-1)}.\qedhere\]
\end{proof}

\begin{theorem}\label{thm:Bastide4/3GP}
    Let~$T$ be a~$k^+$-branching tree on~$n$ vertices, where~$k\ge3$. Then
    \[b(T) \le \left\lceil \sqrt{\frac{4n}{3(k-1)}}~\right\rceil + 2.\]
\end{theorem}
\begin{proof}
    Consider the tree~$T'$ obtained by removing all leaves of~$T$.
    This tree~$T'$ is connected, forms a dominating set of~$T$, and by \Cref{lem:k-branchleaflowerbound},~$|T'|\le n/(k-1)$.
    By \Cref{thm:Bastide},
    \[b(T') \le \left\lceil \sqrt{\frac{4n}{3(k-1)}}~\right\rceil + 1.\]
    We may require an extra round to burn all leaves of~$T$, so~$b(T)\le b(T')+1$; this gives the required result.
\end{proof}

The proof of the following theorem is analogous to that of~\Cref{thm:Bastide4/3GP} when~$k\ge 3$. 
When~$k=2$, the result follows directly from \Cref{thm:Norin}.

\begin{theorem}\label{thm:NorinGP}
    Let~$T$ be a~$k^+$-branching tree on~$n$ vertices, where~$k\ge2$. Then
    \[b(T) \le (1+o(1))\sqrt{n/(k-1)}.\]
\end{theorem}

We show that the bound in \Cref{thm:NorinGP} is asymptotically tight.
We say that a tree is a \emph{caterpillar} if all vertices are within distance~$1$ of a central path.

\begin{observation}\label{obs:GP}
    Let~$k\ge 2$.
    There exists a~$k^+$-branching tree on~$n$ vertices with~$b(T)\ge\left\lceil\sqrt{n/(k-1)}\right\rceil$.
\end{observation}
\begin{proof}
    Consider a~$k^+$-branching caterpillar~$T$ with central path on $t = \left\lfloor\frac{n-2}{k-1} \right\rfloor + 2$ vertices.
    Then, as~$T$ contains~$P_t$ as a subtree,
    \[b(T) \ge b(P_{t}) = \left\lceil\sqrt{t}\right\rceil
    \ge \left\lceil\sqrt{\frac{n-2}{k-1}-1+2}\right\rceil
    \ge \left\lceil\sqrt{n/(k-1)}\right\rceil.\qedhere\]
\end{proof}

To conclude this section, we check when our bound on the burning number for~$k^+$-branching trees (\Cref{thm:k-angBoundle}) matches or outperforms that of Theorem~\ref{thm:Bastide4/3GP}.

\begin{theorem}\label{thm:UsVSBastide}
    Let~$T$ be a~$k^+$-branching tree on~$n$ vertices, where
    \[n\le \displaystyle(k-1){\left [\sqrt{\frac{k-2}{k-1}}- \frac{1}{\sqrt{3}}\right ]^{-2}}.\]
    Then
    \[\sqrt{\frac{4n(k-2)}{(k-1)^2}} \le \sqrt{\frac{4n}{3(k-1)}}+2.\]   
\end{theorem}
\begin{proof}
    We have
    \[\sqrt{\frac{4n(k-2)}{(k-1)^2}} - \sqrt{\frac{4n}{3(k-1)}} = \frac{2\sqrt{n}}{\sqrt{k-1}}\left[ \sqrt{\frac{k-2}{k-1}} - \frac{1}{\sqrt{3}}\right] \le 2.\qedhere\]
\end{proof}
For example, when~$k=3$, 
this occurs whenever~$n\le 118$.
When~$k=4$, we have~$n\le 52$.
The function~$f(k) = (k-1){\left [\sqrt{\frac{k-2}{k-1}}- \frac{1}{\sqrt{3}}\right ]^{-2}}$ has a local minimum at~$k=5$, and it is strictly increasing for~$k\ge 5$.
We illustrate the bounds on~$n$ for some more values of~$k$ in \Cref{tab:UsvsBastide}.
\begin{table}[ht]
\centering
\begin{tabular}{|l||l|l|l|l|l|l|l|l|l|l|l|l|l|}
\hline
$k$ & 3   & 4  & 5  & 6  & 7  & 8  & 9  & 10 & 20 & 50 & 100 & 200 \\ \hline
$n$ & 118 & 52 & 48 & 49 & 53 & 57 & 62 & 67 & 121 & 288 & 567 & 1127  \\ \hline
\end{tabular}
\caption{An illustration of the bound on~$n$ from \Cref{thm:UsVSBastide}.
For a fixed~$k$ value, let~$T$ be a~$k^+$-branching tree on~$n$ vertices.
If~$n$ is at most the value listed in the corresponding column, then the bound for~$b(T)$ from \Cref{thm:k-angBoundle} outperforms or matches that of \Cref{thm:Bastide4/3GP}.}
\label{tab:UsvsBastide}
\end{table}

\section{Discussion}
\NV{By our results, a graph that contains a $k^{+}$-branching spanning tree can be burned efficiently. This naturally raises the question of what graphs allow for such a spanning tree. This question has recently garnered some attention and some sufficient conditions have been formulated (cf. \cite{furuya2025k+branchingspanning}), but much is still unknown. In particular, no structural characterisation of graphs with $k^{+}$-branching spanning trees is currently known. }

\NV{A related question is the following: given a graph $G$, what is the maximal $k$ such that $G$ contains a $k^{+}$-branching spanning tree? If we write $\mathrm{branch}(G)$ for that maximal $k$, we find the following bound on the burning number of $G$:
\[b(G)\leq \min\left\{\left\lceil{\sqrt{\frac{4(\mathrm{branch}(G)-2)|G|}{(\mathrm{branch}(G)-1)^2}}}~\right\rceil,\left\lceil\sqrt{\frac{4|G|}{3(\mathrm{branch}(G)-1)}}~\right\rceil+2\right\}\]
by \Cref{thm:k-angBoundle,thm:Bastide4/3GP}.
Note that~$\mathrm{branch}(G)$ is well-defined as every graph has a spanning tree, i.e.,~$\mathrm{branch}(G)\ge 2$.}
\YM{This means the above bound on~$b(G)$ is also well-defined.}

\NV{A third question of interest is the following. In the class of $k^{+}$-branching trees with $n$ vertices, there is a trivial lower bound for the burning number, since the star graph can be burned in two rounds. However, if we consider $k$-branching trees instead, i.e., trees in which every vertex has either degree $1$ or $k$, this trivial lower bound fails. So, given $n$ and $k$, what are the minimal and maximal burning numbers in the class of $k$-branching trees with $n$ vertices? Similarly, if the burning number and $k$ are given, what can be said about $n$?}


\bibliography{z_bib}

\end{document}